\documentclass[12pt,a4paper,reqno]{amsart}
\usepackage[utf8]{inputenc} 
\usepackage[english]{babel}

\usepackage{amsfonts}
\usepackage{amsmath}
\usepackage{amssymb, amscd}
\usepackage{amsthm}
\usepackage[margin=1in]{geometry}

\newtheorem{theorem}[subsection]{Theorem}
\newtheorem{lemma}[subsection]{Lemma}
\newtheorem{proposition}[subsection]{Proposition}
\newtheorem{definition}[subsection]{Definition}
\newtheorem{remark}[subsection]{Remark}

\newcommand{\comment}[1]{}

\newcommand{\bN}{\mathbb{N}}
\newcommand{\bZ}{\mathbb{Z}} 
\newcommand{\classGroups}{\mathcal{G}}
\newcommand{\alphabet}{\mathcal{A}}
\newcommand{\classAdmWords}{\mathcal{W}}

\newtheorem{cor}[subsection]{Corollary}

\newcommand{\aut}{\operatorname{Aut}}

\newcommand{\dd}{\mathcal D(M)}

\newcommand{\com}{\left[ G\wr_n\bZ, G\wr_n\bZ \right]}

\newcommand\wrm[1]{\mathop{\wr}\limits_{#1}}

\title{First Betti numbers of orbits of Morse functions on surfaces}
\author{Iryna Kuznietsova, Yuliia Soroka}
\address{Department of Algebra and Topology, Institute of Mathematics of NAS of Ukraine, Tereshchenkivska str. 3, Kyiv, 01024, Ukraine}
\curraddr{}
\email{kuznietsova@imath.kiev.ua, sorokayulya@imath.kiev.ua}

\subjclass[2000]{20E22, 20F16, 57T15}
\keywords{Wreath products, Homology groups, Morse functions}

\begin{document}

\begin{abstract}
In this article we study algebraic properties of the specific class of groups~$\classGroups$ generated by direct products and wreath products. Such class of groups appears in calculation of fundamental groups of orbits of Morse functions on compact manifolds. We prove that for any group $G\in\classGroups$ the ranks of the center $Z(G)$ and the quotient by commutator subgroup $G/[G,G]$ coincide. Moreover, this rank is a first Betti number of the orbit of Morse function.

\end{abstract}

\maketitle
\section{Introduction}
	Let $A$ be a group and $n\in\mathbb Z$.
	We will denote by $A\wr_n\bZ$ a semidirect product of $A^n$ and $\bZ$ with respect to the natural action of $\bZ$ on $A^n$ by cyclic shifts of coordinates. More precisely, $A\wr_n\bZ=A^n\rtimes_\varphi \bZ$, where the homomorphism $\varphi\colon\mathbb{Z}\to \aut A^n$ is defined by $\varphi(k) (a_{0},\dots,a_{n-1})= (a_{k \operatorname{ mod }n},\dots, a_{n-1+k \operatorname{ mod } n})$ for all $k\in\mathbb{Z}$, $(a_{0},\dots,a_{n-1})\in A^n$. Such semidirect product $A\wr_n\bZ$ is called the {\it wreath product} of $A$ and $\bZ$.
	
		Note that the following groups are same:
		\begin{align*}
		&A\wrm{1}\bZ=A\times\bZ,& {1}\wrm{n}\bZ=\bZ.	
		\end{align*}
		
\begin{definition}\label{def:classG}	
Let $\classGroups$ be a minimal class of groups satisfying the following conditions:
	
		\begin{enumerate}
			\item[\rm 1)] $ 1 \in \classGroups$;
			\item[\rm 2)]  if $A, B \in \classGroups$, then $A \times B \in \classGroups$;
			\item[\rm 3)] if $ A \in \classGroups$ and $n\geq 1$, then $A \wr_{n} \bZ\in \classGroups$.
		\end{enumerate}
	\end{definition}

In other words a group $G$ belongs to the class $\classGroups$ iff $G$ is obtained from trivial group by a finite number of operations $\times$, $\wr_{n} \bZ$.
	It is easy to see that every group $G\in\classGroups$ can be written as a word in the alphabet $\alphabet=\left\lbrace 1 , \bZ, \left( , \right) , \times, \wr_{2}, \wr_{3}, \wr_{4},\dots\right\rbrace $. We will call such word a {\it presentation} of the group $G$ in the alphabet $\alphabet$. Evidently, the presentation of a group is not uniquely determined. 
	
	For example, there are the following presentations for the same group
	
	\begin{equation*} 
		\left(1\wrm{3}\bZ\right)\times\bZ=\bZ\times\left(1\wrm{3}\bZ\right)=\bZ\times\bZ=1\times \bZ\times\bZ.	
	\end{equation*}
	We will show that the number of symbols $\bZ$ in the presentation of a group $G\in\classGroups$ in the alphabet $\alphabet$ is uniquely determined by $G$. 
	
	Denote by $Z(G)$  and  $[G,G]$  the center and the commutator subgroup of $G$ respectively. The main result of the article is the following theorem.
	
	\begin{theorem}\label{ZC}
	Let $G\in\classGroups $, $\omega$ be an arbitrary presentation of $G$ in the alphabet $\alphabet $, and $\beta_1 (\omega)$ be the number of symbols $\bZ$ in the presentation $\omega$. 	
	 Then there are the following isomorphisms:
	$$Z(G) \cong G/ [G,G]\cong \bZ^{\beta_1(\omega)}.$$
	
	In particular, the number $\beta_1(\omega)$ depends only on the group $G$.
	\end{theorem}
{\bf Geometric interpretation.} 
    The groups from the class $\classGroups$ appear as fundamental groups of orbits of Morse functions on surfaces.

Let $M$ be a compact surface and $\dd$ be the group of $C^\infty$-diffeomorphisms of $M$. There is a natural right action of the group $\dd$ on the space of smooth functions $C^\infty (M,P)$ defined by the rule: $(f,h)\longmapsto f\circ h$, where $h\in\dd$, $f\in C^\infty (M,P)$ and $P$~is a real line $\mathbb{R}$ or a circle $S^1$.

 Let
 $$\mathcal{O}(f)=\{f\circ  h\,\, |\,  h\in\dd\}$$
 be the {\it orbit} of $f$ under the above action.

 Endow the spaces $\dd$, $C^\infty (M,P)$ with Whitney $C^\infty$-topologies.  Let $\mathcal{O}_f(f)$ denote the path component of $f$ in $\mathcal{O}(f)$.

  \comment{
 	Denote by  $\mathcal{F}(M,P)$ the space of smooth functions $f\in C^{\infty}(M,P)$ satisfying the following conditions:
 	\begin{enumerate}
 		\item
 		The function $f$ takes constant value at $\partial M$ and has no critical point in $\partial M$.
 		\item
 		For every critical point $z$ of $f$ there is a local presentation $f_z\colon \mathbb{R}^2 \to \mathbb{R}$ of $f$ near $z$ such that $f_z$ is a homogeneous polynomial  $\mathbb{R}^2 \to \mathbb{R}$ without multiple factors.
 		
 	\end{enumerate}
 	}
 	A map $f\in C^{\infty}(M,P)$ will be called {\it Morse} if all its critical points are non-degenerate. 
 	A Morse map $f$ is {\it generic} if it takes distinct values at distinct critical points.
 
 Homotopy types of stabilizers and orbits of Morse functions were calculated in a series of papers by Sergiy Maksymenko \cite{Maksymenko:AGAG:2006}, \cite{Maksymenko:UMZ:ENG:2012}, Bohdan Feshchenko \cite{Feshchenko:MFAT:2016},\cite{Feshchenko:Zb:2015},  and Elena Kudryav\-tseva \cite{Kudryavtseva:SpecMF:VMU:2012}, \cite {Kudryavtseva:MathNotes:2012}, \cite{Kudryavtseva:MatSb:ENG:2013}, \cite {Kudryavtseva:ENG:DAN2016}.
In particular, E.~Kudryav\-tseva shown that in the case $M\neq S^2$ for each Morse function $f$ there is a free action of some finite group $H$ on $k$-torus $(S^1)^k$ such that $\mathcal{O}_f(f)$ is homotopy equivalent to \- $k$-dimensional torus $(S^1)^k/H$. In fact, it was shown in \cite{Maksymenko:AGAG:2006} by S.~Maksymenko that if $f$ is generic, then $H$ is trivial, so $\pi_n\mathcal{O}_f(f)\simeq\pi_n((S^1)^k)$, and the general case of nontrivial $G$ was described in \cite{Kudryavtseva:SpecMF:VMU:2012}, \cite {Kudryavtseva:MathNotes:2012}, \cite{Kudryavtseva:MatSb:ENG:2013}, \cite {Kudryavtseva:ENG:DAN2016} by E. Kudryavtseva. Furthermore, precise algebraic structure of such groups for the case $M\neq S^2, T^2$ was described in \cite{Maksymenko:UMZ:ENG:2012}. 
The~following theorem relating $\pi_1\mathcal{O}(f)$ with the class $\classGroups$ is a direct consequence of results of \cite{Maksymenko:UMZ:ENG:2012}.
\begin{theorem} \cite{Maksymenko:UMZ:ENG:2012}\label{SI}
Let $M$ be a connected compact oriented surface except 2-sphere and  2-torus and let $f\colon M\to P$ be a Morse function. Then $\pi_1\mathcal{O}(f)\in\classGroups$.
\qed

\end{theorem}
As a consequence of our main result Theorem \ref{ZC} and the previous theorem we get the following.
\begin{cor}
Let $M$ be a connected compact oriented surface distinct from $S^2$ and $T^2$, $f$ be a Morse function on M, $G=\pi_1\mathcal{O}_f(f)\in\classGroups $, $\omega$ be an arbitrary presentation of $G$ in the alphabet $\alphabet $, and $\beta_1 (\omega)$ be the number of symbols $\bZ$ in the presentation $\omega$. Then the first integral homology group $H_1(\mathcal{O}(f), \bZ)$ of $\mathcal{O}(f)$ is a free abelian group of rank $\beta_1(\omega)$:
$$H_1(\mathcal{O}(f), \bZ)\simeq\bZ^{\beta_1(\omega)}.$$

In particular, $\beta_1(\omega)$ is the first Betti number of $\mathcal{O}(f)$.
\end{cor}

\begin{proof}
To prove the corollary we use the well-known Hurewicz theorem, see \cite{Hatcher}, stating that for every path-connected topological space $X$ there is isomorphism 

	$$ H_1(X, \bZ)\simeq\pi_1 X\bigl/[\pi_1 X, \pi_1 X].$$
	
As a consequence of the Hurewicz theorem, Theorem \ref{SI} and Theorem \ref{ZC} we get
$$ H_1(\mathcal{O}(f), \bZ)\simeq\pi_1\mathcal{O}(f)/[\pi_1\mathcal{O}(f), \pi_1\mathcal{O}(f)]=G/[G,G]\simeq \bZ^{\beta_1(\omega)}.$$

\end{proof}



\textbf{Structure of the paper.} 
In $\S 2$ we prove Theorem~\ref{pr: center of wreath product} about centers of wreath products $A\wr_X B$ of arbitrary groups $A$ and $B$ in the case of non-effective action of $B$ on the set $X$. 	Let $G\in\classGroups $, $\omega$ be an arbitrary presentation of $G$. We also show in Theorem~\ref{mth1} that the center of a group $G$ of the class $\classGroups$ is isomorphic to $\bZ^{\beta_1(\omega)}$, it is the first part of Theorem~\ref{ZC}. In $\S 3$ we find the commutator subgroups of any group~$G$, see Theorem~\ref{cs}, and the quotient groups $G\mathop{\wr}_{n}\bZ \bigl/ \com$ (Theorem~\ref{com}). The~second part of Theorem~\ref{ZC} that $G/[G,G]$ is also isomorphic to $ \bZ^{\beta_1(\omega)}$ is established in Theorem~\ref{mth2}.

\section{Centers of wreath products}
Let  $A$ and $B$ be two groups. Suppose there is an action of $B$ on the set $X$. In other words we have the homomorphism $\varphi$ from $B$ to the permutation group $\Sigma(X)$. For $b \in B$ denote by $\varphi_b\colon X\to X$ the corresponding permutation. Let also $Map(X,A)$ be the group of all maps $f: X \rightarrow A$ with respect to the pointwise multiplication. 
Then the group $B$ acts on $Map(X,A)$ by the following rule: the result of the action of $b\in B$ on $f \in Map(X,A)$ is the composition map:
\[
f \circ \varphi_b :  X \longrightarrow X \longrightarrow A.
\]
The semidirect product $Map(X,A) \rtimes_{\varphi} B$ corresponding to this action is called the \textit{unrestricted wreath product} of $A$ and $B$ and denoted by $A \ Wr_{X} B$. 
Hence, it is the Cartesian product $Map(X,A) \times B$ with the multiplication given by the formula 
\[ 	(f_1, b_1) \cdot (f_2, b_2)
= \bigl((f_1 \circ \varphi_{b_2})\cdot f_2, b_1 \cdot b_2 \bigr)
\]
for $(f_1, b_1), (f_2, b_2) \in Map(X,A) \rtimes_\varphi B$.

Denote by $\sigma (f)$ the support of the function $f \in Map(X,A)$:
\[\sigma (f)=\left\lbrace x \in X | f(x) \neq e, \mbox { where } e \mbox{ is a unit of } A\right\rbrace,
\]
and by $Map_{fin}(X,A)$ the subset of $Map(X,A)$ consisting only of functions with a finite support, $\left| \sigma (f)\right|<~\infty$.
The semidirect product $Map_{fin}(X,A) \rtimes_{\varphi} B$ is called the \textit{restricted} wreath product, we denote it by $A \ wr_{X} B$.

Notice, that for the center of the group $A$ we use the notation $Z(A)$. Let $\tilde{D}(A)$ denote the subgroup $Map(X,Z(A))$ of functions $h: X \rightarrow Z(A)$ which are constant on each orbit of action of $B$ on $X$, and let $D(A)$ denote the subgroup $Map_{fin}(X,Z(A))$ of functions with the same property. 

It follows from Theorem~4.2 \cite{Meldrum:1995} that there are isomorphisms:
\[
Z(A\ Wr_{X} B) \cong \tilde{D}(A) ,  \qquad	Z(A\ wr_{X} B) \cong D(A),
\]
where the group $B$ acts on $X$ effectively.

In the case of non-effective action of $B$ on $X$ for arbitrary groups $A,B$ we get more general situation. In this subsection we extend the Theorem~4.2 \cite{Meldrum:1995} and consider the case of non-effective action of $B$ on $X$. 

Denote the set of all the orbits of $B$ on $X$ by $\mathcal{O}$, and the set of finite ones by $\mathcal{O}_{fin}$.
Recall that the direct product indexed by infinite set consists of all infinite sequences, while the direct sum consists only of sequences with finitely many elements distinct from zero. 
\begin{theorem}\label{pr: center of wreath product}
	There are the following isomorphisms:
	\begin{equation}\label{unrestricted product}
	Z(A\ Wr_{X} B) = \tilde{D}(A) \times \left(\ker \varphi \cap Z(B) \right) \cong \prod\limits_{\lambda \in \mathcal{O}} Z(A) \times \left(\ker \varphi \cap Z(B) \right),
	\end{equation} 
	\begin{equation}\label{restricted product}
	Z(A\ wr_{X} B)=D(A) \times \left(\ker \varphi \cap Z(B) \right)\cong \bigoplus\limits_{\lambda \in \mathcal{O}_{fin}} Z(A) \times \left(\ker \varphi \cap Z(B) \right).
	\end{equation}
\end{theorem}

Let $\tilde{Q}$ be a subgroup of $A \ Wr_X B$ whose elements $(f,l)$ satisfy the conditions: 
\begin{enumerate}
	\item [a)] $f$ is constant on each orbit of $B$ on $X$, i.e. $f(x)=a_{\lambda}$ for any $x\in O_{\lambda}$, $O_{\lambda} \in \mathcal{O}$, and every $a_{\lambda}$ is the element of the center $Z(A)$,
	\item [b)]  $l \in \ker \varphi \cap Z(B)$, where $Z(B)$ is the center of $B$.
\end{enumerate}
Obviously, if  an element of $A \ Wr_X B$ satisfies the conditions a) and b), then the element belongs to the center, so $\tilde{Q} \subset Z( A\ Wr_{X} B)$.

For any $y \in X$ and $c\in A$ we define the function $g_{y,c} \in Map(X,A)$ by:

\begin{equation}\label{eq: element g}
g_{y,c}(x)= 
\begin{cases}
c, &\mbox{if } x =y;\\
e, &\mbox{if } x\neq y.
\end{cases}
\end{equation}

Let $S$ be the set of elements $(g_{y,c},p)$ of $A \ Wr_X B$, where $p \in B$.
The set $S$ is also a subset of the restricted product $A \ wr_X B$. Let the centralizer of the set $S$ of $A \ Wr_X B$  and the centralizer of the set $S$ of $A \ wr_X B$ be denoted by $\tilde{C}(S)$ and $C(S)$ respectively. 	 It is clear that $Z(A\ Wr_X B) \subset \tilde{C}(S)$ and $Z(A\ wr_X B) \subset C(S)$. Therefore, for the group $\tilde{Q}$ we have inclusions:
\[
\tilde{Q} \subset Z(A\ Wr_X B) \subset \tilde{C}(S).
\]

\begin{lemma}\label{lm: center of unrestricted product}
	The following identities hold:
	\[
	Z(A\ Wr_X B)=\tilde{C}(S)=\tilde{Q}.
	\]
\end{lemma}
\begin{proof}
	
	To prove the lemma it is enough to check the inclusion $\tilde{C}(S) \subset \tilde{Q}$.

	Assume the element $(g_{y,c},p) \in S$ and $(f,l) \in \tilde{C} (S)$, where $f, g_{y,c} \in Map(X,A)$,~$g_{y,c}$ is defined by (\ref{eq: element g}), and $l, p \in B$. Then, by definition, we get equality	$(f,l)(g_{y,c},p)=~(g_{y,c},p)(f,l)$. Hence 
	\[
	\left(( f \circ \varphi_{p}) \cdot g_{y,c}, lp\right)   =\left( (g_{y,c} \circ \varphi_{l}) \cdot f, pl\right) .
	\]
	
	Therefore, $lp=pl$ for any $p$, so $l \in Z(B)$, and we also get 
	\begin{equation}\label{the product of elelments }
	( f \circ \varphi_{p}(x)) \cdot g_{y,c}(x)   = (g_{y,c} \circ \varphi_{l} (x)) \cdot f(x).
	\end{equation}

For $x\neq y, x\neq\varphi_l^{-1}(y)$ in (\ref{the product of elelments }) we have:
		\begin{equation}\label{the property of center_1}
		f \circ \varphi_{p}(x)=f(x).
		\end{equation}

	The equality (\ref{the property of center_1}) is true for an arbitrary $p \in B$, so  $f$ takes the same value on the whole orbit.
	
	Notice that we can choose $g_{y,c}(x)$ with another fixed element $y$. Therefore $f$ is constant on each orbit of $B$ on $X$, i.e. $f(x)=a_{\lambda} $ for any $x\in O_{\lambda}$, $O_{\lambda} \in \mathcal{O}$.
	
	\comment{Substituting $g_{y,c}(x)$ into (\ref{the product of elelments }), we have:
		\begin{equation}\label{the property of center_1}
		f \circ \varphi_{p}(x)=f(x), \mbox{ for } x\neq y, x\neq\varphi_l^{-1}(y).
		\end{equation}
		Notice that the equality (\ref{the property of center_1}) is true for an arbitrary $p \in B$ and we can consider $g_{y,c}(x)$ with another fixed element $y$. Therefore $f$ is constant on each orbit of $B$ on $X$, i.e. $f(x)=a_{\lambda} $ for any $x\in O_{\lambda}$, $O_{\lambda} \in \mathcal{O}$.
	}
	It remains to show that every $a_{\lambda}$ is the element of the center $Z(A)$ and $l \in \ker \varphi$.
	
	To show this let us analyze the equality (\ref{the product of elelments }) for the case $x=y$.
	There are two cases:
	\begin{enumerate}
		\item [(i)] if $\varphi_l(y)=y$ we get
		\[f(y) g_{y,c}(y)  =g_{y,c}(y)f(y),	\]
		\item [(ii)] if $\varphi_l(y)\neq y$ we get
		\[f(y) g_{y,c}(y)  =f(y).\]
	\end{enumerate}
	The second case is impossible since $g_{y,c}(y)\neq e$.  Hence, $l\in\ker\varphi$.  It follows from the first case that for every $y \in X$  we get $f(y)\in Z(A)$ since $g_{y,c}(y)=c$, where $c$ is an arbitrary element of $A$.
	So, we obtain that the conditions a) and b) hold. 	
	
\end{proof}
\begin{lemma}\label{lm: center of restricted product}
	The center $Z(A \ wr_X B)$ is the intersection of $Z(A\ Wr_X B)$ and $A\ wr_X B$, i.e.
	\[
	Z(A \ wr_X B)=Z(A\ Wr_X B) \cap (A\ wr_X B).
	\]
\end{lemma}
\begin{proof}
	Indeed, the inclusion $\left( Z(A\ Wr_X B) \cap (A\ wr_X B)\right)  \subset Z(A \ wr_X B)$ is obvious.
	
	Let us check the inverse inclusion. Assume $(f,l) \in Z(A \ wr_X B)$. Since $S \subset A \ wr_X B$ we get 
	\[
	Z(A \ wr_X B) \subset C(S) \subset \tilde{C}(S) \cap (A \ wr_X B)  \stackrel{\mbox{{\tiny Lemma~\ref{lm: center of unrestricted product}}}}{=}  Z(A\ Wr_X B) \cap (A\ wr_X B).
	\]
	
\end{proof}
\comment{
	Let $D(A)$ denotes the group $Map(X,Z(A))$ of functions $h: X \rightarrow Z(A)$ which are constant on each orbit of action of $B$ on $X$. In this notation it was proved in Lemma~\ref{lm: center of unrestricted product} that $D(A)\times \left(\ker \varphi \cap Z(B) \right)$ is the subgroup of $A\ Wr_{X}B$ such that
	\[Z(A\ Wr_{X}B) = D(A)\times \left(\ker \varphi \cap Z(B) \right).\]
	Let also $D_{fin}(A)$ be the group $Map_{fin}(X,Z(A))$ of functions $h: X \rightarrow Z(A)$ which are constant on each orbit of action of $B$ on $X$. According to Lemma~\ref{lm: center of restricted product} 
	\[Z(A\ wr_{X}B) = D_{fin}(A)\times \left(\ker \varphi \cap Z(B) \right).\]
}

\begin{proof}[Proof of Theorem \ref{pr: center of wreath product}]
	
If we have the unrestricted product $A \ Wr_{X}B$, then the number of orbits in $Z(A\ Wr_{X}B)$ 
can be infinite. If we have the restricted product  $A\ wr_{X}B$, then there can be only a finite number of finite orbits in $Z(A\ wr_{X}B)$. 
	
	
\comment{	If we have the unrestricted product $A \ Wr_{X}B$, then the number of orbits in $Z(A\ Wr_{X}B)$ 
	 can be infinite, so in this case the cardinal number of $\mathcal O$ will be denoted by $\mbox{card} \Lambda$.
	If we have the restricted product  $A\ wr_{X}B$, then there can be only a finite number of finite orbits in $Z(A\ wr_{X}B)$, we will denote such number by $\mbox{card} \Lambda_{fin}$. }

	According to Lemma~\ref{lm: center of unrestricted product}  and Lemma~\ref{lm: center of restricted product}, there are bijections
	\[\psi_1\colon   Z(A\ Wr_{X}B)\to \prod\limits_{\lambda \in \Lambda} Z(A) \times \left(\ker \varphi \cap Z(B) \right),\]
	\[\psi_2\colon  Z(A\ wr_{X}B)\to \bigoplus\limits_{\lambda \in \Lambda_{fin}} Z(A) \times \left(\ker \varphi \cap Z(B) \right)\]
	defined by $\psi_1(f,l)=(a_1,a_2,\dots,a_{s_1},l)$ and $\psi_2(f,l)=(a_1,a_2,\dots,a_{s_2},l)$. It is easy to check that $\psi_1$ and $\psi_2$ are homomorphisms.
\end{proof}

\begin{remark} If $X=\bZ_n$ we will denote a wreath product of $A$ and $B$  by $A \wr_n B$. In this notation it was proved in Theorem \ref{pr: center of wreath product} that $D(A)\times \left(\ker \varphi \cap Z(B) \right)$ is the subgroup of $A\wr_{n}B$ such that
	\[Z(A\wr_{n}B) = D(A)\times \left(\ker \varphi \cap Z(B) \right).\]
\end{remark}

\begin{cor}\label{centers} 
	\[Z\left(A\mathop{\wr}\limits_{n} \bZ\right)=\left\{(a,a, \dots,a, nk)|a\in Z(A),k\in\bZ\right\} \cong D(A) \times n\bZ \cong Z(A)\times \bZ.
	\]
	\[Z(A \times B) \cong Z(A) \times Z(B).
	\]
\end{cor}
Indeed, for groups $A\mathop{\wr}\limits_{n} \bZ$ of the class $\classGroups$ we have only one orbit of the action of $B$ on~$X$. According to Proposition~\ref{pr: center of wreath product},  we obtain the first equivalence of Corollary~\ref{centers}. The second is obvious. 

For example, \[Z \Big( \big( (\bZ \wr_{3} \bZ)\times (\bZ \wr_{5} \bZ) \big)  \wr_{7} \bZ \Big)  \cong Z \Big( (\bZ \wr_{3} \bZ)\times (\bZ \wr_{5} \bZ)\Big) \times 7\bZ \cong \]
\[\cong Z \Big( (\bZ \wr_{3} \bZ)\times (\bZ \wr_{5} \bZ)\Big) \times \bZ \cong Z\big(\bZ \wr_{3} \bZ\big) \times Z\big(\bZ \wr_{5} \bZ\big) \times \bZ \cong\]
\[ \cong \bZ \times 3\bZ \times \bZ \times 5\bZ \times \bZ \cong \bZ^4 \times \bZ.
\]

\begin{theorem}\label{mth1}
Let $G\in\classGroups$, $\omega$ be any presentation of $G$ in the alphabet $\alphabet $, and $\beta_1 (\omega)$ be the number of symbols $\bZ$ in the presentation $\omega$. Then $Z(G)\simeq\bZ^{\beta_1(\omega)}$.
\end{theorem}
\begin{proof}
The proof follows from Remark \ref{centers} and the induction on the number of symbols 1 and $\bZ$ in the presentation $\omega$ denoted by $l(\omega)$. It will be convenient to denote by $\widetilde{\omega}$ the group from the class $\classGroups$ determined by a word $\omega$. In particular, $\omega$ is a presentation of $\widetilde{\omega}$ in the alphabet $\alphabet $. 
Since $\widetilde{\omega}$ is a group isomorphic to $G$, we have
	\[ Z(G)=Z(\widetilde{\omega}).
	\]
Obviously, if $l(\omega)=1$, then $\omega$ is either 1 or $\bZ$, so $Z(G)\simeq 1$ or $Z(G)\simeq \bZ$ respectively.
	Suppose we proved that $Z(G)\simeq\bZ^{\beta_1(\omega)}$ for all words with $l(\omega)\leq k$. Let us show this for $l(\omega)=k+1$. In this case the presentation $\omega$ is either
\begin{enumerate}	
\item a direct product $\omega_1\times\omega_2$ such that $l(\omega_1)+l(\omega_2)=k+1$, $l(\omega_1)\leq k$, $l(\omega_2)\leq k$, or
 
\item a wreath product $\omega_1 \wr_{n} \bZ$, where $l(\omega_1)=k$.
\end{enumerate}
	According to Remark \ref{centers}, the induction assumption and the evident observation $\beta_1(\omega_1)+\beta_1(\omega_2)={\beta_1(\omega_1\times\omega_2)}$ we get for the first case 
	\[Z(\widetilde{\omega_1} \times \widetilde{\omega_2}) \cong Z(\widetilde{\omega_1}) \times Z(\widetilde{\omega_2})\simeq\bZ^{\beta_1(\omega_1)}\times\bZ^{\beta_1(\omega_2)}\simeq\bZ^{\beta_1(\omega_1)+\beta_1(\omega_2)}\simeq\bZ^{\beta_1(\omega_1\times\omega_2)}\simeq\bZ^{\beta_1(\omega)},
		\] and for the second case
	\[
	Z(\widetilde{\omega})\simeq Z(\widetilde{\omega_1} \wr_{n} \bZ) \simeq Z(\widetilde{\omega_1}) \times \bZ\simeq\bZ^{\beta_1(\omega_1)}\times \bZ\simeq\bZ^{\beta_1(\omega_1\wr_n\bZ)}\simeq\bZ^{\beta_1(\omega)}.
	\]
	
\end{proof}

\section{Commutator subgroup}
\begin{theorem}\label{cs}
For any group $G$ the commutator subgroup of $G\mathop{\wr}\limits_n\bZ$ coincides with the following group 
$$\left[ G\mathop{\wr}\limits_n\bZ, G\mathop{\wr}\limits_n\bZ \right]=\left\{(g_1, g_2, \dots, g_n, 0) | \prod_{i=1}^n g_i\in [G,G]\right \}$$
\end{theorem}

\begin{proof}

Let us first show that every $g=(g_1, g_2, \dots, g_n, 0)$ such that $\prod_{i=1}^n g_i\in [G,G]$ is in $\com$. We will prove that the elements $h_1, h_2$ of the group $G\wr_n\bZ$, 
\[h_1=(g_1, g_2, \dots, g_n, k),\quad h_2=(g_1, g_2, \dots, g_{n-2},g_{n-1}g_n, e,k),\]
lie in the same conjugacy class, i.e. 
\begin{equation}\label{eq: the same conjugacy class}
h_2=h_1f, \mbox{ where } f\in\com.
\end{equation}

 Indeed, $f=(e,e,\dots,e,g_n,g_n^{-1},0)$ satisfies the equality (\ref{eq: the same conjugacy class}). It is easy to check that $f=~[c,d]$, where $c=(e,e,\dots, e, g_n^{-1},1)$,  $d=(e,e,\dots, e, g_n,0)$, and hence $f\in \com$.

   Similarly, by induction, we can obtain that the elements 
   \[h_1=(g_1, g_2, \dots, g_n, k), \ h_3=\left( \prod_{i=1}^n g_i,e,e, \dots,e, k \right) \]
   lie in the same conjugacy class. 
   
   Notice that for elements $\alpha=(a,e,\dots,e,0)$, $\beta=(b,e,e,\dots,e,0)$ from $G\wr_n\bZ$ we have the equality 
   \[[\alpha,\beta]=([a,b], e,e,\dots,e,0).\]
 
   So, every $g=(g_1, g_2, \dots, g_n, 0)$ such that $\prod_{i=1}^n g_i\in [G,G]$ is in the same conjugacy class with the element
    \[\left( \prod_{i=1}^n g_i,e,e, \dots,e, 0)\right) \in\langle  ([a,b], e,e,\dots,e,0)\rangle=\com.\]

Let now $g=(g_1, g_2, \dots, g_n, k)\in\com$. Let also
$a=(a_1,a_2,\dots,a_n,l)$  and $b=(b_1,b_2,\dots, b_n,p)$ 
 be the elements in $ G\wr_n\bZ.$ By straightforward calculation we obtain
\begin{equation}\label{eq: com}
aba^{-1}b^{-1}=(a_{1-l}b_{1-l-p}a^{-1}_{1-l-p}b^{-1}_{1-p},\dots,a_{n-l}b_{n-l-p}a^{-1}_{n-l-p}b^{-1}_{n-p},0).
\end{equation}

 So, evidently, $k=0$. Obviously, in the notation (\ref{eq: com}) each $a_i$, $b_i$, $a^{-1}_i$, $b^{-1}_i$ enters once for each commutator $aba^{-1}b^{-1}$ and its inverse.
  Each $g\in\com$ is generated by commutators with such property, so the product of its first $n$ coordinates has a form 
  \[\prod_{i=1}^n g_i=c_1^{i_1}c_2^{i_2}\cdots c_r^{i_r},  i_r\in \{\pm 1\}, c_i\in G,\]
   where $c_i$ may not be different, but the sum of powers of same elements is always zero. 
   
   Since $c_ic_j=c_jc_i[c_i^{-1}, c_j^{-1}]$ by permutations we cancel out all $c_i$, so only commutators will remain. Thus, we get $\prod_{i=1}^{n}g_i\in [G,G]$.
 
\end{proof}

\begin{theorem} \label{com}
For any group $G$ we have the following isomorphisms of quotient groups
$$G\mathop{\wr}_{n}\bZ \bigl/ \com   \cong G/[G,G] \times \bZ.$$
\end{theorem}
\begin{proof}
Let us construct a homomorphism $\varphi\colon G\wr_n\bZ\to G / [G,G] \times \bZ$ defined by 

$$\varphi(g_1, g_2, \dots, g_n, k)=\left(\left(\prod_{i=1}^n g_i \right )[G,G], k\right).$$

To check that $\varphi$ is a homomorphism we compute 
\begin{gather*}
\varphi(a_1, a_2, \dots, a_n, k)\varphi(b_1, b_2, \dots, b_n, p)=(a_1 a_2\cdots a_n[G,G], k)(b_1 b_2 \cdots b_n[G,G], p)=\\
=(a_1 a_2\cdots a_nb_1 b_2 \cdots b_n[G,G], kp),\\
\varphi((a_1, a_2, \dots, a_n, k)(b_1, b_2, \dots, b_n, p))=\varphi(a_{(1+p) \operatorname{ mod }n}b_1, a_{(2+p) \operatorname{ mod }n}b_2, \dots, a_{(n+p) \operatorname{ mod }n}b_n, kp)=\\
=(a_{(1+p) \operatorname{ mod }n}b_1 a_{(2+p) \operatorname{ mod }n}b_2 \cdots a_{(n+p) \operatorname{ mod }n}b_n[G,G], kp).\\
\end{gather*}

Since $G/[G,G]$ is abelian we get 
$$ (a_{(1+p) \operatorname{ mod }n}b_1 a_{(2+p) \operatorname{ mod }n}b_2 \cdots a_{(n+p) \operatorname{ mod }n}b_n[G,G], kp)=(a_1 a_2\cdots a_nb_1 b_2 \cdots b_n[G,G], kp),$$
hence $\varphi$ is a homomorphism.
 
It is onto since for each element $(h,n)\in G/ [G,G] \times \bZ$ there is $(h,e,e,\cdots,n)$ which satisfies 
$$\varphi(h,e,e,\dots,n)=(h,n).$$
 The kernel of $\varphi$ is then $\ker \varphi=\{(g_1, g_2, \dots, g_n, 0) | \prod_{i=1}^n g_i\in [G,G]\}$, which coincides with $\com$.  
\end{proof}

\begin{theorem}\label{mth2}
Let $G\in\classGroups$, $\omega$ be any presentation of $G$ in the alphabet $\alphabet $, and $\beta_1 (\omega)$ be the number of symbols $\bZ$ in the presentation $\omega$. Then $G/[G,G]\simeq\bZ^{\beta_1(\omega)}$.
\end{theorem}
\begin{proof}
The proof is similar to Theorem \ref{mth1}. One should only use Theorem \ref{com} and the fact that for any two groups $A$ and $B$ it holds
\begin{equation}\label{dir}
 A\times B/[A\times B, A\times B]\simeq A/[A, A]\times B/[B, B]
\end{equation}
instead of Remark \ref{centers}.
To prove (\ref{dir}) it is enough to check that the map $\varphi\colon A\times B\to A/[A, A]\times B/[B, B]$ defined by 
$$ (a,b)\mapsto(a[A,A], b[B,B])$$
is a surjective homomorphism with the kernel $\ker \varphi=[A\times B, A\times B]$. 

Since $\varphi$ is a product of surjective homomorphisms $\varphi_1$, $\varphi_2$ defined by  
\begin{gather*}
\varphi_1\colon A\times B\to A/[A,A],\qquad\varphi_1(a,b)=a[A,A], \\
\varphi_2\colon A\times B\to B/[B,B],\qquad \varphi_2(a,b)= b[B,B],
 \end{gather*}
it is a surjective homomorphism as well. 

Further, notice that $$\ker\varphi=\{(a,b)|a\in[A,A],b\in[B,B]\}.$$

Let us check that $\ker\varphi\subset[A\times B, A\times B].$
Indeed, let $(\prod[a_i, b_i], \prod [c_j,d_j])\in\ker\varphi$, where $a_i, b_i\in A$, $c_j, d_j\in B$, then 
\begin{gather*}
(\prod_i[a_i, b_i], \prod_j [c_j,d_j])=(\prod_i[a_i, b_i], e)(e, \prod_j [c_j,d_j])=\prod_i([a_i, b_i], e)\prod_j(e,  [c_j,d_j])=\\
=\prod_i[(a_i,e),(b_i,e)]\prod_j[(e,c_j), (e,d_j)]
\in[A\times B, A\times B].
\end{gather*}
Conversely, for any commutator $[(a,b),(c,d)]$ in $[A\times B, A\times B]$ we have 
$$[(a,b),(c,d)]=(a,b)(c,d)(a^{-1},b^{-1})(c^{-1},d^{-1})=([a,c],[b,d])\in\ker\varphi.$$
Theorem is proved.
\comment{
It can easily be proven in the same way as Theorem \ref{mth1}, using the induction on $\beta_1 (\omega)$.
Instead of Remark \ref{centers} we have now Theorem \ref{com} and the fact that 
$$ G_1\times G_2/[G_1\times G_2, G_1\times G_2]\simeq G_1/[G_1, G_1]\times G_2/[G_2, G_2].$$
}
\end{proof}

Now we can get the evident proof of our main result, Theorem\ref{ZC}. Let us recall it.\\

	{\bf Theorem.\ref{ZC}}
	{\it Let $G\in\classGroups $, $\omega$ be an arbitrary presentation of $G$ in the alphabet $\alphabet $, and $\beta_1 (\omega)$ be the number of symbols $\bZ$ in the presentation $\omega$. 	
	 Then there are the following isomorphisms:
	$$Z(G) \cong G/ [G,G]\cong \bZ^{\beta_1(\omega)}.$$}
\begin{proof} 
Under the same assumptions, we get in Theorem \ref{mth1} that $Z(G)\simeq\bZ^{\beta_1(\omega)}$, and in Theorem \ref{mth2} that $G/[G,G]\simeq\bZ^{\beta_1(\omega)}$. Thus, evidently,
	$$Z(G) \cong G/ [G,G]\cong \bZ^{\beta_1(\omega)}.$$
\end{proof}



\comment
{\bibliographystyle{plain} 
\bibliography{a1}

\end {document}
}
\comment{

}

\bibliographystyle{amsalpha} 
\bibliography{a1}

\def\cprime{$'$}
\providecommand{\bysame}{\leavevmode\hbox to3em{\hrulefill}\thinspace}
\providecommand{\MR}{\relax\ifhmode\unskip\space\fi MR }
\providecommand{\MRhref}[2]{%
  \href{http://www.ams.org/mathscinet-getitem?mr=#1}{#2}
}
\providecommand{\href}[2]{#2}
\begin{thebibliography}{Kud12b}

\bibitem[Fes15]{Feshchenko:Zb:2015}
B.~G. Feshchenko, \emph{Deformation of smooth functions on $2$-torus whose
  {K}ronrod-{R}eeb graphs is a tree}, Topology of maps of low-dimensional
  manifolds, vol.~12, Pr. Inst. Mat. Nats. Akad. Nauk Ukr. Mat. Zastos., no.~6,
  Nats\=\i onal. Akad. Nauk Ukra\"\i ni, \=Inst. Mat., Kiev, 2015,
  pp.~204--219.

\bibitem[Fes16]{Feshchenko:MFAT:2016}
Bohdan Feshchenko, \emph{Actions of finite groups and smooth functions on
  surfaces}, Methods Funct. Anal. Topology \textbf{22} (2016), no.~3, 210--219.
  \MR{3554649}

\bibitem[Hat02]{Hatcher}
Allen Hatcher, \emph{Algebraic topology}, Cambridge University Press,
  Cambridge, 2002. \MR{1867354}

\bibitem[Kud12a]{Kudryavtseva:SpecMF:VMU:2012}
E.~A. Kudryavtseva, \emph{Special framed {M}orse functions on surfaces},
  Vestnik Moskov. Univ. Ser. I Mat. Mekh. (2012), no.~4, 14--20. \MR{3026876}

\bibitem[Kud12b]{Kudryavtseva:MathNotes:2012}
\bysame, \emph{The topology of spaces of {M}orse functions on surfaces}, Math.
  Notes \textbf{92} (2012), no.~1-2, 219--236, Translation of Mat. Zametki
  {{\bf{9}}2} (2012), no. 2, 241--261. \MR{3201559}

\bibitem[Kud13]{Kudryavtseva:MatSb:ENG:2013}
\bysame, \emph{On the homotopy type of spaces of {M}orse functions on
  surfaces}, Sb. Math. \textbf{204} (2013), no.~1, 75--113.

\bibitem[Kud16]{Kudryavtseva:ENG:DAN2016}
\bysame, \emph{Topology of spaces of functions with prescribed singularities on
  the surfaces}, Dokl. Akad. Nauk \textbf{93} (2016), no.~3, 264--266.
  \MR{3527003}

\bibitem[Mak06]{Maksymenko:AGAG:2006}
S.~I. Maksymenko, \emph{Homotopy types of stabilizers and orbits of {M}orse
  functions on surfaces}, Ann. Global Anal. Geom. \textbf{29} (2006), no.~3,
  241--285. \MR{MR2248072 (2007k:57067)}

\bibitem[Mak12]{Maksymenko:UMZ:ENG:2012}
\bysame, \emph{Homotopy types of right stabilizers and orbits of smooth
  functions on surfaces}, Ukrainian Math. Journal \textbf{64} (2012), no.~9,
  1186--1203 (Russian).

\bibitem[Mel95]{Meldrum:1995}
J.~D.~P. Meldrum, \emph{Wreath products of groups and semigroups}, Pitman
  Monographs and Surveys in Pure and Applied Mathematics, vol.~74, Longman,
  Harlow, 1995. \MR{1379113}

\end{thebibliography}


\begin{thebibliography}{10}
\bibitem{Meldrum:1995}
Meldrum, J. D. P., \emph{Wreath products of groups and semigroups}, Pitman Monographs and Surveys in Pure and Applied Mathematics, \textbf{74} (1995), 324.

\bibitem{Maksymenko:UMZ:ENG:2012}
S.~Maksymenko.
\newblock Homotopy types of right stabilizers and orbits of smooth functions on
  surfaces.
\newblock {\em Ukrainian Math. Journal}, 64(9):1186--1203, 2012.
\bibitem{Maksymenko:AGAG:2006}
S.~Maksymenko.
\newblock Homotopy types of stabilizers and orbits of {M}orse functions on
  surfaces.
\newblock {\em Ann. Global Anal. Geom.}, 29(3):241--285, 2006.
\end{thebibliography}

\end{document}